\documentclass[11pt,reqno]{amsart}

\usepackage{amsmath, amsthm, amscd, amssymb, amsfonts, amsbsy}

\usepackage{esint}

\usepackage[bitstream-charter]{mathdesign}
\usepackage[T1]{fontenc}
\usepackage{xcolor} 
\usepackage[pdftex]{hyperref}
\newtheorem{lemma}{Lemma}

\newtheorem{remark}{Remark}
\newtheorem{assumption}{Assumption}
\newtheorem{theorem}{Theorem}

\newcommand{\N}{\mathbb{N}}
\newcommand{\R}{\mathbb{R}}

\newcommand{\bR}{\mathbb{R}}
\newcommand{\bu}{\mathbf{u}}
\newcommand{\bv}{\mathbf{v}}
\newcommand{\bw}{\mathbf{w}}

\newcommand{\norm}[1]{{\left\Vert #1 \right\Vert}}
\newcommand{\set}[1]{{\left\lbrace #1 \right\rbrace}}
\newcommand{\lr}[1]{{\left( #1 \right)}}

\newcommand{\intR}{\int_{\R^3}}

\newcommand{\ep}{\varepsilon}
\newcommand{\na}{\nabla}

\newcommand{\ti}{\to \infty}
\newcommand{\dx}{\, dx}

\newcommand{\qand}{\quad \text{ and } \quad}
\newcommand{\qif}{\quad \text{ if } \quad}

\newcommand{\qon}{\quad \text{ on }}

\newcommand{\qfs}{\quad \text{ for some }}

\DeclareMathOperator*{\BMO}{BMO}

\begin{document}
\baselineskip=18pt

\title[Liouville-Type Theorems]{Refined Liouville-Type Theorems for the Stationary Navier--Stokes Equations}

\author{Youseung Cho \& Minsuk Yang}

\address{
Yonsei University, Department of Mathematics, Yonseiro 50, Seodaemungu, Seoul 03722, Republic of Korea
}
\email{(Cho) youseung@yonsei.ac.kr \& (Yang) m.yang@yonsei.ac.kr}

\begin{abstract}
We study smooth solutions to the three-dimensional stationary Navier--Stokes equations and establish new Liouville-type theorems under refined decay assumptions.
Building on the work of Cho et al., we introduce a refinement to previously known integrability criteria and analyze the associated averaged quantities.
Our main result shows that if the $L^p$ growth rate of a solution remains bounded for some $3/2 < p < 3$, then the solution must be trivial.
The proof combines averaged decay estimates, energy inequalities, and an iteration scheme.
\\
\\
\noindent{\bf AMS Subject Classification Number:} 
35Q30, 
35B53, 
76D03
\\
\noindent{\bf keywords:} 
Navier--Stokes equations, 
Liouville-type theorem, 
Energy estimates
\end{abstract}

\maketitle

\section{Introduction}
\label{S1}

We study smooth solutions to the three-dimensional stationary Navier--Stokes equations 
\begin{align}
\label{E11}
\begin{split}
-\Delta \bu + (\bu \cdot \na) \bu + \na \pi &= 0, \\
\nabla \cdot \bu &= 0,
\end{split}
\end{align}
where $\bu$ denotes the velocity field and $\pi$ the pressure.
Under additional assumptions, such solutions reduce to trivial states, i.e., identically zero or constant. 
We establish Liouville-type theorems of this form using the energy method.
Throughout the paper, we use the notation $B(\rho) = \set{x \in \R^3 : |x| < \rho}$ for $0 < \rho < \infty$ and 
\[
A(\rho_1,\rho_2) = \set{x \in \R^3 : \rho_1 \le |x| < \rho_2},
\quad 
0 < \rho_1 < \rho_2 < \infty.
\]

Galdi \cite{MR2808162} showed that any solution $\bu$ to \eqref{E11} is trivial provided $\bu \in L^{\frac{9}{2}}(\bR^3)$.
The key step is that the decay condition
\begin{equation}
\label{E12}
\lim_{\rho \ti} \frac{1}{\rho} \int_{A(\rho,2\rho)} |\bu|^3 = 0,
\end{equation}
implies $\int_{\R^3} |\na \bu|^2 = 0$, and \eqref{E12} follows from H\"older's inequality under $L^{9/2}$-integrability.
Seregin and Wang \cite{MR3937507} established the inequality
\begin{equation}
\label{E13}
\int_{\bR^3} |\nabla \bu|^2 
\le C_{p,l} \liminf_{\rho \to \infty} \frac{\norm{\bu}^3_{L^{p,l}(A(\rho,2\rho))}}{\rho^{\frac{9}{p}-2}}
\end{equation}
for $p > 3$, $3 \le l \le \infty$ (or $p=l=3$), and further obtained decay criteria in Lorentz spaces, including
\[
\liminf_{\rho \to \infty} \frac{\norm{\bu}_{L^p(A(\rho,2\rho))}}{\rho^{\gamma}} < \infty
\qfs \frac{12}{5} < p < 3 \qand \gamma < \frac{2}{p} - \frac{1}{3}.
\]
Tsai \cite{MR4354995} proved that $\bu = 0$ if 
\[
\liminf_{\rho \to \infty} \frac{\norm{\bu}_{L^p(A(\rho,2\rho))}}{\rho^{\frac{2}{p}-\frac{1}{3}}} = 0
\qfs \frac{12}{5} \le p \le 3.
\]
More recently, Cho et al. \cite{MR4699113} showed that $\bu = 0$ under the general condition
\[
\liminf_{\rho \to \infty} \frac{\norm{\bu}_{L^p(A(\rho,2\rho))}}{\rho^{\frac{2}{p}-\frac{1}{3}}} < \infty
\qfs \frac{3}{2} < p < 3.
\]

One can consider averaged growth conditions of the form
\[
\left( \int_{A(\theta \rho, \theta^{-1} \rho)} |\bu|^p \dx \right)^{\frac{1}{p}} \le \rho^{\frac{2}{p}-\frac{1}{3}} g(\rho)^{\frac{3}{p}-1},
\]
where $0 < \theta < 1$ specifies the averaging region and $g(\rho)$ represents the additional growth.
Many existing Liouville-type results can be formulated in this framework.
Since $\rho^{\frac{2}{p}-\frac{1}{3}}$ represents the sharpest known growth rate, we introduce the factor $g(\rho)^{\frac{3}{p}-1}$, with the exponent chosen to simplify subsequent estimates.

In this paper we refine admissible growth rates by introducing an additional growth function $g$ and formulating Liouville-type criteria in terms of $g$. 
We impose the following conditions.

\begin{assumption}
\label{A1}
Let $g$ be a function from $[1,\infty)$ to $[1,\infty)$ such that 
\begin{enumerate}
\item
$g$ is non-decreasing on $[1,\infty)$,
\item
\begin{equation}
\label{E15}
\lim_{\rho \to \infty} \rho^{-1/3} g(\rho) = 0,
\end{equation}
\item
For all $a \ge 1$,
\begin{equation}
\label{E16}
\int_a^{\infty} \frac{1}{g(\rho)} \frac{d\rho}{\rho} = \infty.
\end{equation}
\end{enumerate}
\end{assumption}

Our main results are as follows.

\begin{theorem}
\label{T1}
Let $g$ satisfy Assumption \ref{A1}.
Suppose $\bu$ is a smooth solution to \eqref{E11} such that 
\begin{enumerate}
\item
$\na \bu \in L^2(\R^3)$,
\item
there exist $0 < \theta < 1$ and $1 \le p < 3$ satisfying 
\begin{equation}
\label{E17}
\limsup_{\rho \ti} \frac{\norm{\bu}_{L^p(A(\theta \rho, \theta^{-1} \rho))}}{\rho^{\frac{2}{p}-\frac{1}{3}} g(\rho)^{\frac{3}{p}-1}} < \infty.
\end{equation}
\end{enumerate}
Then $\bu = 0$.
\end{theorem}

For $\tfrac{3}{2} < p < 3$, the assumption  $\nabla \bu \in L^2(\R^3)$ is not required, although demonstrating this is nontrivial.
We state this result separately as a theorem.

\begin{theorem}
\label{T2}
Let $g$ satisfy Assumption \ref{A1}.
Suppose $\bu$ is a smooth solution to \eqref{E11} such that there exist $0 < \theta < 1$ and $\tfrac{3}{2} < p < 3$ satisfying \eqref{E17}.
Then $\bu = 0$.
\end{theorem}

\begin{remark}
There are many examples satisfying Assumption \ref{A1}. 
An example of $g$ satisfying \eqref{E15} and \eqref{E16} is obtained by setting $\rho_1=1$, defining $\log \rho_{n+1}=(\rho_n)^\alpha$ for $n \in \N$ and $0 < \alpha < 1/3$, and letting $g(\rho)=(\rho_n)^\alpha$ for $\rho_n \le \rho < \rho_{n+1}$.
When $\alpha=\tfrac{1}{3}$, \eqref{E16} holds, but $\limsup_{\rho \to \infty} \rho^{-1/3} g(\rho)=1$. 
Further examples include logarithmic functions such as 
\[
\log (e+\rho), \quad \log (e+\log(e+\rho)), \quad \log (e+\log (e+\log(e+\rho))), \quad \cdots.
\]
\end{remark}

Related results have also been obtained under assumptions on a potential function $V$ with $\bu = \nabla \cdot V$.
Seregin \cite{MR3538409} treated the case of skew-symmetric $V \in \BMO(\R^3)$.
Chae and Wolf \cite{MR3959933} and Cho et al. \cite{MR4572397} extended this framework by formulating growth conditions on $V$ over large balls, subsuming \cite{MR3538409} as a special case.
More recently, Bang and Yang \cite{bang2024saintvenant} refined these criteria using logarithmic factors.
For alternative approaches, see \cite{MR3162482}, \cite{MR4929556}, \cite{MR3548261}, \cite{MR3571910}, and the references therein.

Organization of the paper.
Section \ref{S2} introduces notation and preliminary lemmas, including the definition and derivative of the energy function.
Section \ref{S3} establishes local energy estimates, among them one derived via an approximation scheme.
Sections \ref{S4} and \ref{S5} present the proofs of Theorems \ref{T1} and \ref{T2}, respectively.

\section{Preliminaries}
\label{S2}

We begin with the notation used throughout the paper.
\begin{itemize}
\item
For tensor-valued functions $F$ and $G$,
\[
F:G = \sum_{i,j} F_{ij} G_{ji}.
\]
\item
Generic positive constants are denoted by $C$, with dependence on a parameter $q$ indicated by $C_q$.
\item
We write $a \lesssim b$ if there exists $C>0$ such that $|a| \le C|b|$, and $a \lesssim_q b$ if the constant depends on $q$.
\item
We denote by $|\Omega|$ the Lebesgue measure of a measurable set $\Omega$.
\item
The average of $f$ over $\Omega$ is denoted by 
\[
(f)_\Omega = \fint_\Omega f = \frac{1}{|\Omega|} \int_\Omega f.
\]
\end{itemize}

We next introduce the energy function $E_\theta(\rho)$ and establish its basic properties.

\begin{lemma}
\label{L1}
Let $0 < \theta < 1$.
For $\rho > 0$, define 
\[
E_{\theta}(\rho) = \intR |\na \bu(x)|^2 \eta_{\theta} \lr{\frac{|x|}{\rho}} \dx,
\]
where $\eta_{\theta} : [0,\infty) \to [0,1]$ is given by 
\begin{align*}
\eta_{\theta}(t) = 
\begin{cases}
1 &\qif 0 \le t < \theta, \\
\frac{1-t}{1-\theta} &\qif \theta \le t < 1, \\
0 &\qif 1 \le t < \infty.
\end{cases} 
\end{align*}
Then $E_{\theta}(\rho) \in C^1((0,\infty))$ and for all $\rho > 0$,
\[
E'_{\theta}(\rho) 
= 
\frac{1}{\rho^2 (1-\theta)}
\int_{A(\theta \rho,\rho)} 
|\na \bu(x)|^2 |x| \dx.
\]
In particular, 
\[
E'_{\theta}(\rho) 
\ge 
\frac{\theta}{\rho (1-\theta)}
\norm{\na \bu}_{L^2(A(\theta \rho,\rho))}^2.
\]
\end{lemma}

\begin{proof}
Fix $\rho > 0$ and let $h > 0$.
To compute the right derivative of $E_{\theta}$ at $\rho$, we consider 
\[
\frac{E_{\theta}(\rho (1+h)) - E_{\theta}(\rho)}{\rho h} 
= 
\frac{1}{\rho h} \intR |\na \bu(x)|^2 
\lr{\eta_{\theta}\left(\frac{|x|}{\rho (1+h)}\right) 
- \eta_{\theta}\left(\frac{|x|}{\rho}\right)} \dx.
\]
If $0 < h < 1/\theta - 1$, then by the definition of $\eta_{\theta}$,
\begin{align*}
\eta_{\theta}\lr{\frac{|x|}{\rho(1+h)}} 
- \eta_{\theta}\lr{\frac{|x|}{\rho}} 
= 
\begin{cases}
0 &\qif 0 \le \frac{|x|}{\rho} < \theta, \\
1 - \frac{1-\frac{|x|}{\rho}}{1-\theta} &\qif \theta \le \frac{|x|}{\rho} < \theta (1 + h), \\
\frac{1-\frac{|x|}{\rho (1+h)}}{1-\theta} - \frac{1-\frac{|x|}{\rho}}{1-\theta} &\qif \theta (1 + h) \le \frac{|x|}{\rho} < 1, \\
\frac{1-\frac{|x|}{\rho (1+h)}}{1-\theta} &\qif 1 \le \frac{|x|}{\rho} < 1 + h,\\
0 &\qif 1 + h \le \frac{|x|}{\rho} < \infty.
\end{cases} 
\end{align*}
Thus, we split the above integral into three parts as 
\[
\frac{E_{\theta}(\rho (1+h)) - E_{\theta}(\rho)}{\rho h} = I + II + III,
\]
where 
\begin{align*}
I(h)
&= \frac{1}{\rho h} 
\int_{\set{\theta \le \frac{|x|}{\rho} < \theta (1 + h)}} 
|\na \bu(x)|^2 
\frac{\frac{|x|}{\rho}-\theta}{1-\theta} \dx, \\
II(h)
&= \frac{1}{\rho h} 
\int_{\set{\theta (1 + h) \le \frac{|x|}{\rho} < 1}} 
|\na \bu(x)|^2 
\frac{\frac{|x|}{\rho}-\frac{|x|}{\rho (1+h)}}{1-\theta} \dx, \\
III(h)
&= \frac{1}{\rho h} 
\int_{\set{1 \le \frac{|x|}{\rho} < 1 + h}} 
|\na \bu(x)|^2 
\frac{1-\frac{|x|}{\rho (1+h)}}{1-\theta} \dx.
\end{align*}
It is easily checked that $I(h) \to 0$ and $III(h) \to 0$ as $h \to 0$.
The right derivative of $E$ at $\rho$ is obtained as follows 
\begin{align*}
\lim_{h \to 0} II(h) 
&= 
\lim_{h \to 0} \frac{1}{\rho^2 (1-\theta) (1+h)}
\int_{\set{\theta (1 + h) \le \frac{|x|}{\rho} < 1}} 
|\na \bu(x)|^2 |x| \dx
 \\
&= 
\frac{1}{\rho^2 (1-\theta)}
\int_{A(\theta \rho,\rho)} 
|\na \bu(x)|^2 |x| \dx.
\end{align*}
The left derivative is obtained similarly, so $E_{\theta} \in C^1$ and the formula for $E'_{\theta}(\rho)$ holds.
\end{proof}

We end this section with a standard lemma used to eliminate the pressure term from the energy estimates.

For $1 < q < \infty$, denote $L^q_0(\Omega) = \set{f \in L^q (\Omega) : (f)_\Omega = 0}$ and $W^{1,q}_0(\Omega)$ the closure of $C_c^{\infty}(\Omega)$ in the Sobolev space $W^{1,q}(\Omega)$.

\begin{lemma}
[Lemma 3 of \cite{MR4354995}]
\label{L2}
Let $0 < \theta < 1$, $R > 1$, and $1 < q < \infty$.
There is a linear map $T : L^q_0(A(\theta R, R)) \to W^{1,q}_0(A(\theta R, R))$ such that for $f \in L^q_0(A(\theta R, R))$, the vector field $T f \in W^{1,q}_0(A(\theta R, R))$ satisfies 
\[
\nabla \cdot T f = f 
\qand 
\norm{\na T f}_{L^q(A(\theta R, R))} \leq C_{\theta,q} \norm{f}_{L^q(A(\theta R, R))}.
\]
Moreover, $C_{\theta,q} \to \infty$ as $\theta \to 1$.
\end{lemma}

\section{Energy estimates}
\label{S3}

In this section we collect several lemmas derived from local energy estimates.

\begin{lemma}
\label{L3}
Let $0 < \theta < 1$.
For $\rho > 1$, define 
\[
K_\theta(\rho) = \rho^{-1} \norm{\bu}_{L^3(A(\theta \rho, \theta^{-1} \rho))}^3.
\]
\begin{enumerate}
\item
If $\liminf_{\rho \to \infty} K_\theta(\rho) = 0$, then $\bu = 0$.
\item
If $\liminf_{\rho \to \infty} K_\theta(\rho) < \infty$, then $\nabla \bu \in L^2(\R^3)$ and $\bu \in L^6(\R^3)$.
\end{enumerate}
\end{lemma}

\begin{proof}
Fix $0 < \theta < 1$ and $\rho > 1$.
Let $\rho \leq r < R \leq \theta^{-1} \rho$.
Let $\phi \in C^2_c(B(R))$ be radially decreasing and nonnegative with $\phi = 1$ on $B(r)$ and $(R-r) |\na \phi| + (R-r)^2 |\na^2 \phi| \le C$ independent of $r, R$.
Denote 
\[
S = A(\theta R, R),
\]
so that $A(r,R) \subset S \subset B(R)$.
Since $\nabla \cdot \bu = 0$, we have $\bu \cdot \na \phi^2 \in L^q_0(S)$ for all $1 < q < \infty$ by the divergence theorem.
Define 
\[
\bv = T(\bu \cdot \na \phi^2)
\qon S,
\]
with $T$ the map from Lemma \ref{L2}.
Then $\nabla \cdot \bv = \bu \cdot \na \phi^2$ and for $1 < q < \infty$,
\begin{equation}
\label{E31}
\norm{\na \bv}_{L^q(S)} 
\lesssim_{\theta, q} 
(R-r)^{-1} \norm{\bu}_{L^q(S)}.
\end{equation}
Multiplying the first equation of \eqref{E11} by $\bu \phi^2 - \bv$ and integrating by parts eliminates the pressure term and yields 
\begin{align*}
\int |\na \bu|^2 \phi^2 
&= \frac{1}{2} \int |\bu|^2 \Delta \phi^2
+ \int \na \bu : (\na \bv)^T \\
&\quad + \frac{1}{2} \int |\bu|^2 \bu \cdot \na \phi^2 
- \int (\bu \otimes \bu) : (\na \bv)^T.
\end{align*}
By Sobolev's inequality and integration by parts,
\[
\left( \int |\bu \phi|^6 \right)^{\frac{1}{3}} 
\lesssim 
\int |\na (\bu \phi)|^2  
= \int |\na \bu|^2 \phi^2 
+ \int |\bu|^2 |\na \phi|^2 
- \frac{1}{2} \int |\bu|^2 \Delta \phi^2.
\]
Hence, 
\begin{align*}
\norm{\bu}_{L^6(B(r))}^2 
+ \norm{\na \bu}_{L^2(B(r))}^2 
&\le 
\left( \int |\bu \phi|^6 \right)^{\frac{1}{3}} 
+
\int |\na \bu|^2 \phi^2 \\
&\lesssim 
\int |\na \bu|^2 \phi^2 
+ (R-r)^{-2} \norm{\bu}_{L^2(S)}^2.
\end{align*}
Collecting the estimates and estimating each term using H\"older's inequality gives
\begin{align*}
\norm{\bu}_{L^6(B(r))}^2 
+ \norm{\na \bu}_{L^2(B(r))}^2 
&\lesssim (R-r)^{-2} \norm{\bu}_{L^2(S)}^2 
+ \norm{\na \bu}_{L^2(S)} \norm{\na \bv}_{L^2(S)} \\
&\quad + (R-r)^{-1} \norm{\bu}_{L^{3}(S)}^3
+ \norm{\bu}_{L^{3}(S)}^2 \norm{\na \bv}_{L^3(S)} \\ 
&=: I + II + III + IV.
\end{align*}
By Young's inequality
\begin{align*}
I
&\lesssim (R-r)^{-2} R \norm{\bu}_{L^3(S)}^2 \\
&\lesssim (R-r)^{-1} \norm{\bu}_{L^3(S)}^3
+ (R-r)^{-4} R^3.
\end{align*}
For any $\ep > 0$ there exists $C_\ep > 0$ such that 
\begin{align*}
II
&\lesssim 
(R-r)^{-1} \norm{\na \bu}_{L^2(S)} 
\norm{\bu}_{L^2(S)} \\
&\le \ep \norm{\na \bu}_{L^2(S)}^2 
+ C_\ep (R-r)^{-2} \norm{\bu}_{L^2(S)}^2 \\
&\le \ep \norm{\na \bu}_{L^2(B(R))}^2 
+ C_\ep I
\end{align*}
by using \eqref{E31} with $q=2$ and Young's inequality.
Similarly, using \eqref{E31} with $q=3$, we get 
\[
IV 
\lesssim_{\theta} 
(R-r)^{-1} \norm{\bu}_{L^{3}(S)}^3.
\]
Combining these estimates, we obtain that for $\rho \leq r < R \leq \theta^{-1} \rho$,
\[
\norm{\bu}_{L^6(B(r))}^2 
+ \norm{\na \bu}_{L^2(B(r))}^2 
\lesssim \ep \norm{\na \bu}_{L^2(B(R))}^2
+ 
C_{\ep, \theta} (R-r)^{-1} \norm{\bu}_{L^3(S)}^3
+ 
C_{\ep, \theta} (R-r)^{-4} \rho^3.
\]
Using $S \subset A(\theta \rho, \theta^{-1} \rho)$, choosing $\ep$ small, and applying an iteration argument (see e.g. Lemma 8 of \cite{MR3720841}) yields 
\[
\norm{\bu}_{L^6(B(r))}^2 
+
\norm{\na \bu}_{L^2(B(r))}^2 
\lesssim_{\theta} 
(R-r)^{-1} \norm{\bu}_{L^3(A(\theta \rho, \theta^{-1} \rho))}^3
+ (R-r)^{-4} \rho^3.
\]
Taking $r = \rho$ and $R = \theta^{-1} \rho$ gives for all $\rho > 1$, 
\[
\norm{\bu}_{L^6(B(\rho))}^2 
+
\norm{\na \bu}_{L^2(B(\rho))}^2 
\lesssim_\theta 
K_\theta(\rho) + \rho^{-1}.
\]

If $\liminf_{\rho \to \infty} K_\theta(\rho) = 0$, then $\bu = 0$.

If $\liminf_{\rho \to \infty} K_\theta(\rho) < \infty$, then $\nabla \bu \in L^2(\R^3)$ and $\bu \in L^6(\R^3)$.
\end{proof}

In the following lemma, including the $L^6$ norm of $\bu$ further simplifies subsequent estimates.
The proof proceeds analogously to Lemma \ref{L3}.

\begin{lemma}
\label{L4}
Let $0 < \theta < 1$.
If $\frac{3}{2} < p < 3$, then for all $\rho > 1$, there exists $C_{\theta} > 0$ such that 
\[
\norm{\bu}_{L^6(B(\rho))}^2
+ \norm{\na \bu}_{L^2(B(\rho))}^2 
\le C_{\theta}
\rho^{-\frac{6-p}{2p-3}} 
\norm{\bu}_{L^p(A(\theta \rho, \theta^{-1} \rho))}^{\frac{3p}{2p-3}} 
+ 
C_{\theta}
\rho^{-1}.
\]
\end{lemma}

\begin{proof}
After the same estimates in the proof of the previous lemma, we obtain that for any $\ep > 0$ there exists $C_{\ep, \theta} > 0$ such that for $\rho \leq r < R \leq \theta^{-1} \rho$,
\[
\norm{\bu}_{L^6(B(r))}^2 
+ \norm{\na \bu}_{L^2(B(r))}^2 
\lesssim \ep \norm{\na \bu}_{L^2(B(R))}^2
+ 
C_{\ep, \theta} (R-r)^{-1} \norm{\bu}_{L^3(S)}^3
+ 
C_{\ep, \theta} (R-r)^{-4} \rho^3.
\]
If $\frac{3}{2} < p < 3$, then for any $\ep > 0$ there exists $C_\ep > 0$ such that 
\begin{align*}
(R-r)^{-1}
\norm{\bu}_{L^3(S)}^3
&\lesssim (R-r)^{-1}
\norm{\bu}_{L^p(S)}^{\frac{3p}{6-p}}
\norm{\bu}_{L^6(S)}^{\frac{18-6p}{6-p}} \\
&\le \ep \norm{\bu}_{L^6(S)}^2 
+ C_\ep (R-r)^{-\frac{6-p}{2p-3}} 
\norm{\bu}_{L^p(S)}^{\frac{3p}{2p-3}}
\end{align*}
by an interpolation inequality and Young's inequality.
Combining these estimates and using $S \subset A(\theta \rho, \theta^{-1} \rho)$, we obtain that for $\rho \leq r < R \leq \theta^{-1} \rho$,
\begin{align*}
\norm{\bu}_{L^6(B(r))}^2
+ \norm{\na \bu}_{L^2(B(r))}^2 
&\lesssim \ep (\norm{\bu}_{L^6(B(R))}^2
+ \norm{\na \bu}_{L^2(B(R))}^2) \\
&\quad + 
C_{\ep, \theta} (R-r)^{-\frac{6-p}{2p-3}} 
\norm{\bu}_{L^p(A(\theta \rho, \theta^{-1} \rho))}^{\frac{3p}{2p-3}} 
+ 
C_{\ep, \theta} (R-r)^{-4} \rho^3.
\end{align*}
Choosing $\ep$ small and applying an iteration argument (see e.g. Lemma 8 of \cite{MR3720841}) yields 
\[
\norm{\bu}_{L^6(B(r))}^2
+ \norm{\na \bu}_{L^2(B(r))}^2 
\lesssim_{\theta} 
(R-r)^{-\frac{6-p}{2p-3}} 
\norm{\bu}_{L^p(A(\theta \rho, \theta^{-1} \rho))}^{\frac{3p}{2p-3}} 
+ 
(R-r)^{-4} \rho^3.
\]
Taking $r = \rho$ and $R = \theta^{-1} \rho$ gives the result.
\end{proof}

We now derive a critical energy inequality for $E_\theta(\rho)$.
Since $\eta_\theta(|x|/\rho)$ in Lemma \ref{L1} is not differentiable, we proceed by approximation.

\begin{lemma}
\label{L5}
Let $0 < \theta < 1$.
Then for all $\rho > 1$ there exists $C_{\theta} > 0$ such that 
\[
E_{\theta}(\rho) 
\le C_{\theta}
\rho^{-1}
\norm{\na \bu}_{L^2(A(\theta \rho,\rho))}
\norm{\bu}_{L^2(A(\theta \rho,\rho))}  
+ 
C_{\theta}
\rho^{-1} 
\norm{\bu}_{L^3(A(\theta \rho,\rho))}^3.
\]
\end{lemma}

\begin{proof}
Fix $0 < \theta < 1$.
Let $\psi \in C^1_c(B(1))$ be radially decreasing and nonnegative with $\psi = 1$ on $B(\theta)$, and set $\psi_\rho(x) = \psi(x/\rho)$.
Since $\nabla \cdot \bu = 0$, we have $\bu \cdot \na \psi_\rho \in L^q_0(A(\theta \rho,\rho))$ for $1 < q < \infty$ by the divergence theorem.
Let $T$ be the operator from Lemma \ref{L2} and define $\bw = T(\bu \cdot \na \psi_\rho)$.
Then $\nabla \cdot \bw = \bu \cdot \na \psi_\rho$ and for $1 < q < \infty$,
\[
\norm{\na \bw}_{L^q(A(\theta \rho,\rho))} \lesssim_{q, \theta} \rho^{-1} \norm{\bu}_{L^q(A(\theta \rho,\rho))}.
\]
Multiplying the first equation of \eqref{E11} by $\bu \psi_\rho - \bw$ and integrating, we obtain 
\begin{align*}
\int |\na \bu|^2 \psi_\rho 
&= -\int \na \bu :(\bu \otimes \na \psi_\rho) 
+ \int \na \bu : (\na \bw)^T \\
&\quad + \frac{1}{2} \int |\bu|^2 \bu \cdot \na \psi_\rho 
- \int (\bu \otimes \bu) : (\na \bw)^T.
\end{align*}
Applying H\"older's inequality and the bound on $\nabla \bw$ gives
\begin{equation}
\label{E33}
\int |\na \bu|^2 \psi_\rho 
\lesssim_{\theta} 
\rho^{-1} 
\norm{\na \bu}_{L^2(A(\theta \rho,\rho))}
\norm{\bu}_{L^2(A(\theta \rho,\rho))}
+ 
\rho^{-1} 
\norm{\bu}_{L^{3}(A(\theta \rho,\rho))}^3.
\end{equation}
For $n > \frac{2}{1-\theta}$, define $\eta_{n, \theta} \in C^1_c ([0,\infty))$ by 
\begin{align*}
\eta_{n, \theta}(t)
= 
\begin{cases}
1 - \frac{1}{n(1-\theta)}
&\qif 0 \le t < \theta , \\
-\frac{n}{2(1-\theta)} \left( t - \theta \right)^2 + 1 - \frac{1}{n(1-\theta)}
&\qif \theta \le t < \theta + \frac{1}{n}, \\
-\frac{1}{1-\theta} (t-1) - \frac{1}{2n(1-\theta)}
&\qif \theta + \frac{1}{n} \le t < 1 - \frac{1}{n}, \\
\frac{n}{2(1-\theta)} (t-1)^2 
&\qif 1 - \frac{1}{n} \le t < 1,\\
0 
&\qif 1 \le t < \infty.
\end{cases} 
\end{align*}
Then $\eta_{n, \theta} \to \eta_{\theta}$ uniformly, and hence 
\[
E_{\theta}(\rho) 
= \int |\na \bu(x)|^2 \eta_{\theta} \lr{\frac{|x|}{\rho}} \dx 
= \lim_{n \ti} \int |\na \bu(x)|^2 \eta_{n,\theta} \lr{\frac{|x|}{\rho}} \dx.
\]
Using this and \eqref{E33} with $\psi_{\rho}(x) = \frac{n(1-\theta)}{n(1-\theta)-1} \eta_{n, \theta} (|x|/ \rho)$, we get the desired result.
\end{proof}

We next estimate the growth of solutions over annular regions.

\begin{lemma}
\label{L6}
Let $0 < \theta < 1$, $\rho > 1$, and $g : [1, \infty) \to [1,\infty)$.
If there exists $\frac{3}{2} < p < 3$ such that
\[
\norm{\bu}_{L^p(A(\theta \rho, \theta^{-1} \rho))} < \rho^{\frac{2}{p}-\frac{1}{3}} g(\rho)^{\frac{3}{p}-1},
\]
then there exists $C_{\theta} > 0$ such that  
\[
\norm{\bu}_{L^6(B(\rho))}
+ \norm{\na \bu}_{L^2(B(\rho))}
\le C_{\theta} g(\rho)^{\frac{9-3p}{4p-6}}.
\]
\end{lemma}

\begin{proof}
From the assumptions and Lemma \ref{L4}, 
\begin{align*}
\norm{\bu}_{L^6(B(\rho))}^2
+ \norm{\na \bu}_{L^2(B(\rho))}^2 
&\lesssim_\theta 
\rho^{-\frac{6-p}{2p-3}} 
\norm{\bu}_{L^p(A(\theta \rho, \theta^{-1} \rho))}^{\frac{3p}{2p-3}} 
+ 
\rho^{-1} \\
&\lesssim g(\rho)^{\frac{9-3p}{2p-3}} + \rho^{-1} \\
&\lesssim g(\rho)^{\frac{9-3p}{2p-3}}.
\end{align*}
Taking square roots gives the desired bound.
\end{proof}

\section{Proof of Theorem \ref{T1}}
\label{S4}

By Lemma \ref{L3}, it suffices to prove 
\[
\liminf_{\rho \to \infty} K_\theta(\rho) = 0.
\]
Applying the Gagliardo--Nirenberg inequality, 
\begin{align*}
K_\theta(\rho)^{\frac{1}{3}} 
&= 
\rho^{-\frac{1}{3}} \norm{\bu}_{L^3(A(\theta \rho, \theta^{-1} \rho))} \\
&\lesssim_{\theta} 
\rho^{-\frac{1}{3}} 
\norm{\bu}_{L^p(A(\theta \rho, \theta^{-1} \rho))}^{\frac{p}{6-p}}
\norm{\na \bu}_{L^2(A(\theta \rho, \theta^{-1} \rho))}^{\frac{6-2p}{6-p}}
+ \rho^{\frac{2}{3}-\frac{3}{p}} \norm{\bu}_{L^p(A(\theta \rho, \theta^{-1} \rho))}.
\end{align*}
From \eqref{E17}, for all sufficiently large $\rho$,
\[
\norm{\bu}_{L^p(A(\theta \rho, \theta^{-1} \rho))}
\lesssim \rho^{\frac{2}{p}-\frac{1}{3}} g(\rho)^{\frac{3}{p}-1}.
\]
Combining the two estimates gives 
\[
K_\theta(\rho)^{\frac{1}{3}} 
\lesssim_{\theta} 
\left(g(\rho) \norm{\na \bu}_{L^2(A(\theta \rho, \theta^{-1} \rho))}^2\right)^{\frac{3-p}{6-p}}
+ \left(\rho^{-\frac{1}{3}} g(\rho)\right)^{\frac{3}{p}-1}.
\]
Since $\lim_{\rho \to \infty} \rho^{-\frac{1}{3}} g(\rho) = 0$ from \eqref{E15}, it remains to show 
\[
\liminf_{\rho \ti} \, (g(\rho) \norm{\na \bu}_{L^2(A(\theta \rho, \theta^{-1} \rho))}^2) = 0.
\]
Suppose, for contradiction, that $\liminf_{\rho \ti} (g(\rho) \norm{\na \bu}_{L^2(A(\theta \rho, \theta^{-1} \rho))}^2) > 0$.
Then there exists $C > 0$ such that for all large $\rho$,
\[
\frac{C}{g(\rho)} < \norm{\na \bu}_{L^2(A(\theta \rho, \theta^{-1} \rho))}^2.
\]
Since $g$ is non-decreasing, for all large $N \in \N$, we have 
\begin{align*}
\int_{\theta^{-2N}}^\infty \frac{1}{g(r)} \frac{dr}{r}
&= 
\sum_{n = N}^\infty \int_{\theta^{-2n}}^{\theta^{-(2n+2)}} \frac{1}{g(r)} \frac{dr}{r} \\
&\le 
\sum_{n = N}^\infty \frac{1}{g(\theta^{-2n})} \int_{\theta^{-2n}}^{\theta^{-(2n+2)}} \frac{dr}{r} \\ 
&=
-2 \log \theta \sum_{n = N}^\infty \frac{1}{g(\theta^{-2n})} \\
&< 
\frac{-2 \log \theta}{C} \sum_{n = N}^\infty 
\norm{\na \bu}_{L^2(A(\theta^{-(2n-1)}, \theta^{-(2n+1)}))}^2 \\
&\le 
\frac{-2 \log \theta}{C} \norm{\na \bu}_{L^2(\R^3)}^2 < \infty,
\end{align*}
which contradicts \eqref{E16} in Assumption \ref{A1}.
This completes the proof of Theorem \ref{T1}.
\qed

\section{Proof of Theorem \ref{T2}}
\label{S5}

Theorem \ref{T2} follows from the next lemma.
Indeed, if Lemma \ref{L7} holds, then combining \eqref{E17} with Lemma \ref{L3} yields $\nabla \bu \in L^2(\R^3)$, and Theorem \ref{T1} implies $\bu \equiv 0$.

\begin{lemma}
\label{L7}
Let $g$ satisfy Assumption \ref{A1}.
For all $0 < \theta < 1$ and $3/2 < p < 3$, we have
\begin{equation}
\label{E50}
\liminf_{\rho \to \infty} \frac{\norm{\bu}_{L^3(A(\theta \rho, \theta^{-1} \rho))}}{\rho^{\frac{1}{3}}}
\le 
\limsup_{\rho \ti} \frac{\norm{\bu}_{L^p(A(\theta \rho, \theta^{-1} \rho))}}{\rho^{\frac{2}{p}-\frac{1}{3}} g(\rho)^{\frac{3}{p}-1}}.
\end{equation}
\end{lemma}

\begin{proof}
\noindent
\textbf{Contradiction setup.}
Suppose there exist $0 < \theta < 1$, $3/2 < p < 3$, and $0 < M < \infty$ such that 
\[
\limsup_{\rho \ti} \frac{\norm{\bu}_{L^p(A(\theta \rho, \theta^{-1} \rho))}}{\rho^{\frac{2}{p}-\frac{1}{3}} g(\rho)^{\frac{3}{p}-1}}
< 
M 
< 
\liminf_{\rho \to \infty} \frac{\norm{\bu}_{L^3(A(\theta \rho, \theta^{-1} \rho))}}{\rho^{\frac{1}{3}}}.
\]
Then for all sufficiently large $\rho$, 
\begin{equation}
\label{E52}
\norm{\bu}_{L^p(A(\theta \rho, \theta^{-1} \rho))}
< M \rho^{\frac{2}{p}-\frac{1}{3}} g(\rho)^{\frac{3}{p}-1}
< \norm{\bu}_{L^3(A(\theta \rho, \theta^{-1} \rho))} \rho^{\frac{2}{p}-\frac{2}{3}} g(\rho)^{\frac{3}{p}-1}.
\end{equation}
Since 
\[
\norm{\bu}_{L^3(A(\theta \rho, \theta^{-1} \rho))} \le C_1 \rho^{1/2} \norm{\bu}_{L^6(A(\theta \rho, \theta^{-1} \rho))}
\]
with $C_1 = |A(\theta, \theta^{-1})|^{1/6}$, we deduce that for all large $\rho$, 
\[
\frac{\rho^{\frac{1}{2}-\frac{3}{p}} \norm{\bu}_{L^p(A(\theta \rho, \theta^{-1} \rho))}}{C_1(\rho^{-\frac{1}{3}} g(\rho))^{\frac{3}{p}-1}} 
< \frac{M}{C_1} \rho^{-\frac{1}{6}}
\le
\norm{\bu}_{L^6(A(\theta \rho, \theta^{-1} \rho))}.
\]
By the Gagliardo--Nirenberg inequality, there exists $C_2 = C_{\theta} > 0$ such that
\[
\norm{\bu}_{L^6(A(\theta \rho, \theta^{-1} \rho))} 
\le C_2 \norm{\na \bu}_{L^2(A(\theta \rho, \theta^{-1} \rho))} 
+ C_2 \rho^{\frac{1}{2}-\frac{3}{p}} \norm{\bu}_{L^p(A(\theta \rho, \theta^{-1} \rho))}. 
\]
From \eqref{E15}, we have 
\[
(\rho^{-\frac{1}{3}} g(\rho))^{\frac{3}{p}-1} \le \frac{1}{2 C_1  C_2}
\] 
for all sufficiently large $\rho$.
Combining estimates, we have for all sufficiently large $\rho$ that
\[
\norm{\bu}_{L^6(A(\theta \rho, \theta^{-1} \rho))} 
\le 2C_2 \norm{\na \bu}_{L^2(A(\theta \rho, \theta^{-1} \rho))} 
\]
Hence, for all sufficiently large $\rho$, 
\begin{equation}
\label{E51}
0 < \frac{M}{C_1 \rho^{1/6}}
< \norm{\bu}_{L^6(A(\theta \rho, \theta^{-1} \rho))} 
\le 2 C_2 \norm{\na \bu}_{L^2(A(\theta \rho, \theta^{-1} \rho))}.
\end{equation}
By Lemma \ref{L5}, we have
\begin{align*}
E_{\theta^2}(\theta^{-1} \rho)
&\lesssim_{\theta} 
\rho^{-1} 
\norm{\bu}_{L^2(A(\theta \rho, \theta^{-1} \rho))} 
\norm{\na \bu}_{L^2(A(\theta \rho, \theta^{-1} \rho))} 
+ 
\rho^{-1} 
\norm{\bu}_{L^3(A(\theta \rho, \theta^{-1} \rho))}^3 \\
&=: 
I + II.
\end{align*}
If we establish that for all large $\rho$,
\begin{equation}
\label{E53}
I + II
\lesssim_{\theta, p, M} 
g(\theta^{-1} \rho)^{\frac{9-3p}{6-p}}
\norm{\na \bu}_{L^2(A(\theta \rho, \theta^{-1} \rho))}^{\frac{18-6p}{6-p}},
\end{equation}
then, Lemma \ref{L1} yields that for all sufficiently large $\rho$, 
\begin{equation}
\label{E54}
E_{\theta^2}(\rho)
\lesssim_{\theta, p, M} \left( \rho g(\rho) E'_{\theta^2}(\rho) \right)^{9-3p \over 6-p}.
\end{equation}
Note that from \eqref{E51}, we have $E_{\theta^2} (\rho) > 0$ for sufficiently large $\rho$ so that
\[
\frac{1}{g(\rho)} \frac{1}{\rho}
\lesssim_{\theta, p, M} 
\frac{E'_{\theta^2}(\rho)}{E_{\theta^2}(\rho)^{{6-p \over 9-3p}}}.
\]
Integrating over $(a,\infty)$ for large $a$ contradicts \eqref{E16}, since ${6-p \over 9-3p} > 1$.
This proves \eqref{E50}.

\noindent
\textbf{Proof of claim \eqref{E53}.}
It remains to show \eqref{E53}.
The estimate of $I$ depends on $p$.
\begin{itemize}
\item
If $3/2 < p < 2$, then an interpolation inequality gives
\[
\norm{\bu}_{L^2(A(\theta \rho, \theta^{-1} \rho))} 
\le 
\norm{\bu}_{L^p(A(\theta \rho, \theta^{-1} \rho))}^{\frac{2p}{6-p}} 
\norm{\bu}_{L^6(A(\theta \rho, \theta^{-1} \rho))}^{\frac{6-3p}{6-p}}.
\]
For all sufficiently large $\rho$, we have 
\[
\rho^{-\frac{1}{6}} 
\lesssim_{\theta} \norm{\bu}_{L^6(A(\theta \rho, \theta^{-1} \rho))}
\lesssim_{\theta} \norm{\nabla \bu}_{L^2(A(\theta \rho, \theta^{-1} \rho))}
\]
from \eqref{E51} so that
\begin{align*}
\norm{\bu}_{L^2(A(\theta \rho, \theta^{-1} \rho))} 
&\lesssim_{\theta}
\rho^{\frac{3-p}{3(6-p)}}
\norm{\bu}_{L^p(A(\theta \rho, \theta^{-1} \rho))}^{\frac{2p}{6-p}} 
\norm{\bu}_{L^6(A(\theta \rho, \theta^{-1} \rho))}^{\frac{12-5p}{6-p}} \\
&\lesssim_{\theta}
\rho^{\frac{3-p}{3(6-p)}}
\norm{\bu}_{L^p(A(\theta \rho, \theta^{-1} \rho))}^{\frac{2p}{6-p}} 
\norm{\nabla \bu}_{L^2(A(\theta \rho, \theta^{-1} \rho))}^{\frac{12-5p}{6-p}}.
\end{align*}
Thus, by \eqref{E52}, for all sufficiently large $\rho$, we have
\begin{align*}
I
&=
\rho^{-1} 
\norm{\bu}_{L^2(A(\theta \rho, \theta^{-1} \rho))} 
\norm{\na \bu}_{L^2(A(\theta \rho, \theta^{-1} \rho))} \\
&\lesssim_{\theta} 
\rho^{-1} \rho^{\frac{3-p}{3(6-p)}}
\norm{\bu}_{L^p(A(\theta \rho, \theta^{-1} \rho))}^{\frac{2p}{6-p}} 
\norm{\na \bu}_{L^2(A(\theta \rho, \theta^{-1} \rho))}^{\frac{18-6p}{6-p}} \\
&\lesssim_{p, M}  
g(\rho)^{\frac{9-3p}{6-p}}
\norm{\na \bu}_{L^2(A(\theta \rho, \theta^{-1} \rho))}^{\frac{18-6p}{6-p}} \\
&\lesssim 
g(\theta^{-1} \rho)^{\frac{9-3p}{6-p}}
\norm{\na \bu}_{L^2(A(\theta \rho, \theta^{-1} \rho))}^{\frac{18-6p}{6-p}},
\end{align*}
where we used the assumption $g$ is non decreasing in the last inequality.
\item
If $2 \le p \le 12/5$, Jensen's inequality and \eqref{E51} give for all sufficiently large $\rho$,
\begin{align*}
I
&= \rho^{-1} 
\norm{\bu}_{L^2(A(\theta \rho, \theta^{-1} \rho))}^{\frac{4p-6}{6-p}} 
\norm{\bu}_{L^2(A(\theta \rho, \theta^{-1} \rho))}^{\frac{12-5p}{6-p}} 
\norm{\na \bu}_{L^2(A(\theta \rho, \theta^{-1} \rho))} \\
&\lesssim_{p, \theta} \rho^{-2 + \frac{3}{p}} 
\norm{\bu}_{L^p(A(\theta \rho, \theta^{-1} \rho))}^{\frac{4p-6}{6-p}} 
\norm{\bu}_{L^6(A(\theta \rho, \theta^{-1} \rho))}^{\frac{12-5p}{6-p}} 
\norm{\na \bu}_{L^2(A(\theta \rho, \theta^{-1} \rho))} \\
&\lesssim_{M,p,\theta}
\rho^{-\frac{2}{3} + \frac{1}{p}}
g(\rho)^{\frac{(4p-6)(3-p)}{p(6-p)}}
\norm{\na \bu}_{L^2(A(\theta \rho, \theta^{-1} \rho))}^{\frac{18-6p}{6-p}}.
\end{align*}
Since $g$ is non decreasing, we have
\[
I
\lesssim_{M,p,\theta}
\rho^{-\frac{2}{3} + \frac{1}{p}}
g(\theta^{-1} \rho)^{\frac{(4p-6)(3-p)}{p(6-p)}}
\norm{\na \bu}_{L^2(A(\theta \rho, \theta^{-1} \rho))}^{\frac{18-6p}{6-p}},
\]
which implies \eqref{E53} since $-\frac{2}{3} + \frac{1}{p} < 0$ and $\frac{(4p-6)(3-p)}{p(6-p)} < \frac{9-3p}{6-p}$.
\item
If $12/5 < p < 3$, we have $\frac{18-6p}{6-p} < 1$. 
By Lemma \ref{L6}, we have for all sufficiently large $\rho$,
\[
\norm{\na \bu}_{L^2(B(\rho))}
\lesssim_{M,p, \theta}  g(\rho)^{\frac{9-3p}{4p-6}}.
\]
Using Jensen's inequality and above estimate, we have for all sufficiently large $\rho$,
\begin{align*}
I
&\lesssim_{p, \theta} \rho^{-\frac{3}{p} + \frac{1}{2}}
\norm{\bu}_{L^p(A(\theta \rho, \theta^{-1} \rho))}
\norm{\na \bu}_{L^2(A(\theta \rho, \theta^{-1} \rho))}^{\frac{18-6p}{6-p}}
\norm{\na \bu}_{L^2(B(\theta^{-1} \rho))}^{\frac{5p-12}{6-p}} \\
&\lesssim_{M, p, \theta} \rho^{-\frac{1}{p} + \frac{1}{6}} 
g(\theta^{-1} \rho)^{\frac{3-p}{p} + \frac{(5p-12)(9-3p)}{(6-p)(4p-6)}} 
\norm{\na \bu}_{L^2(A(\theta \rho, \theta^{-1} \rho))}^{\frac{18-6p}{6-p}}.
\end{align*}
Since $-\frac{1}{p} + \frac{1}{6} < 0$ and $\frac{3-p}{p} + \frac{(5p-12)(9-3p)}{(6-p)(4p-6)} < \frac{9-3p}{6-p}$, \eqref{E53} holds.
\end{itemize}

Finally, interpolation with \eqref{E51} shows that for all large $\rho$,
\begin{align*}
\norm{\bu}_{L^3(A(\theta \rho, \theta^{-1} \rho))}^3
&\lesssim 
\norm{\bu}_{L^p(A(\theta \rho, \theta^{-1} \rho))}^{3p \over 6-p} 
\norm{\bu}_{L^6(A(\theta \rho, \theta^{-1} \rho))}^{18-6p \over 6-p} \\
&\lesssim_{\theta} 
\norm{\bu}_{L^p(A(\theta \rho, \theta^{-1} \rho))}^{3p \over 6-p} 
\norm{\na \bu}_{L^2(A(\theta \rho, \theta^{-1} \rho))}^{18-6p \over 6-p}.
\end{align*}
Thus, from \eqref{E52} along with the non decreasing assumption on $g$, we have for all large $\rho$, 
\[
II
= \rho^{-1} 
\norm{\bu}_{L^3(A(\theta \rho, \theta^{-1} \rho))}^3 
\lesssim_{\theta, p, M} 
g(\theta^{-1} \rho)^{\frac{9-3p}{6-p}}
\norm{\na \bu}_{L^2(A(\theta \rho, \theta^{-1} \rho))}^{18-6p \over 6-p}.
\]
\end{proof}

\section*{Acknowledgement}

The work was supported by the National Research Foundation of Korea(NRF) (RS-2021-NR059843) and (RS-2024-00442483).

%
%
%

\end{document}